\documentclass[12pt]{article}

\usepackage[english]{babel}
\usepackage{graphicx}
\usepackage{latexsym}
\usepackage{amsfonts}
\usepackage{dsfont}
\usepackage{amscd}
\usepackage{amssymb}
\usepackage{amsmath}
\usepackage{hyperref}
\usepackage{mathrsfs}
\usepackage{amsthm}
\usepackage{tikz}
\usetikzlibrary{graphs}

\input xy
\xyoption{all}

\newcommand{\NN}{\mathbb{N}}

\newcommand{\ZZ}{\mathbb{Z}}

\newcommand{\p}{\mathcal{P}}

\theoremstyle{definition}
\newtheorem{theo}{Theorem}[section]
\newtheorem{rmk}[theo]{Remark}

\newtheorem{ex}[theo]{Example}
\newtheorem{lem}[theo]{Lemma}
\newtheorem{defn}[theo]{Definition}

\newtheorem{prop}[theo]{Proposition}
\date{}
\begin{document}
\title{Bratteli-Vershik models for partial actions of $\ZZ$}
\author{Thierry Giordano\thanks{Research supported by NSERC, Canada.}, Daniel Gon\c calves\thanks{Partially supported by CNPq, Brazil.}, Charles Starling\thanks{Supported by the NSERC grants of Beno\^it Collins, Thierry Giordano, and Vladimir Pestov.}}

\maketitle

\begin{abstract}
Let $U$ and $V$ be open subsets of the Cantor set with finite disjoint complements, and let $h:U\to V$ be a homeomorphism with dense orbits. Building on the ideas of Herman, Putnam, and Skau, we show that the partial action induced by $h$ can be realized as the Vershik map on a Bratteli diagram, and that any two such diagrams are equivalent.
\end{abstract}

\section{Introduction}

A Bratteli-Vershik model is a combinatorial realization of a dynamical system formed by a map (now called the {\em Vershik map}) acting on the path space of an infinite graph called a {\em Bratteli diagram}. Such models were introduced in the study of measurable dynamics by Vershik \cite{V81, V82} who showed that any ergodic automorphism has such a model. Building on the work of Vershik, an analogous result was obtained in the topological setting by Herman, Putnam and Skau: any minimal homeomorphism of the Cantor set has a Bratteli-Vershik model, see \cite{HPS}. The ideas in \cite{HPS} led to the classification of Cantor minimal systems up to orbit equivalence, see \cite{GMPS, GMPS10,GPS1, GPS04}. Also, other important results in the topological setting were obtained: in \cite{GjerdeJohansend} Bratteli-Vershik models were developed for interval exchange transformations and used to study conjugation of these system as well as orbit equivalence to Sturmian systems. Substitutional dynamical systems were studied through their Bratteli-Vershik model in \cite{DuranHostSkau} and Bratteli-Vershik models for Cantor aperiodic systems were developed in \cite{BDM, Medynets}. Very recently a category structure for ordered Bratteli diagrams was proposed such that isomorphism in this category coincides with Herman, Putnam, and Skau's notion of equivalence, see \cite{AminiElliottGolestani}.

As defined, the Vershik map is naturally only a partially defined map. The edges of the Bratteli diagram are given a partial order which extends to a partial order on the finite paths beginning at the root, and the Vershik map sends an infinite path to its successor induced by this ordering. This map is only defined on the paths with at least one nonmaximal edge, which is an open subset of the space of infinite paths. In many of the works above, the constructed model can be arranged to produce only one maximal path and one minimal path, and the Vershik model can hence be extended to a globally defined homeomorphism by sending the maximal path to the minimal path. 

Of course, there are examples of ordered Bratteli diagrams which have a different number of maximal and minimal paths (see \cite{BKY14, BY13}, also see Example \ref{2max1min}) -- in this case extending to a globally defined homeomorphism is impossible. In cases such as this, we still have that the Vershik map is a homeomorphism between open subsets of the path space, and we can hence study it as a {\em partial action} of $\ZZ$ (the Vershik map has been studied as a partially defined map before in the literature, see \cite{GoncalvesRoyer, Mingo}).

Partial actions were originally defined in \cite{Exel} and gradually gained importance as many C*-algebras and algebras were realized as partial crossed products (approximately finite, Bunce-Deddens and Cuntz-Krieger algebras, graph algebras, Leavitt path algebras, algebras associated with integral domains, and self-similar graph algebras among others, see \cite{ExelBoava, KarlsenLarsen, Exel1, Exel2, Exel3, ES16, GoncalvesRoyer1}). In the topological category a partial action of $\ZZ$ provides the correct setting to study the dynamics of a partial homeomorphism: given a homeomorphism $h:U\rightarrow V$ between two open sets of the topological space $X$ one considers the iterates $h^n$, with $n\in \ZZ$, restricted to the appropriate domains. This is the approach taken in \cite{Exel, Exel4}. 

It is our goal in this paper to link the theory of minimal partial actions of $\ZZ$ on the Cantor set with Bratteli diagrams. Our main result, Theorem \ref{maintheo} states that any minimal homeomorphism between open subsets of the Cantor set (whose complements are finite and disjoint) has a Bratteli-Vershik model. In our model the points in the complement of the open sets are identified with maximal and minimal paths in the Bratteli diagram. 

We divide our work in three sections. In Section 2 we recall background material on Bratteli diagrams and partial actions. In Section 3 we construct our Bratteli-Vershik model by first carefully developing a suitable ``first return time'' (Proposition \ref{returntime}), using that to construct an ordered Bratteli diagram (Proposition \ref{KRsequence}), and showing that the resulting Vershik map is isomorphic to our original homeomorphism (Theorem \ref{maintheo}).

\section{Background}

\subsection{Bratteli diagrams}

In this section we recall relevant concepts regarding Bratteli diagrams. We begin with the definition, as introduced by Bratteli in \cite{B72}

\begin{defn}
A {\em Bratteli diagram} is an infinite directed graph $(V,E)$. The vertex set $V$ is the union of a sequence of finite, nonempty, pairwise disjoint sets $V_n$, $n\geq 0$. The set $V_0$ is assumed to consist of a single vertex $v_0$ called the {\em root}. Similarly, the edge set is the union of a sequence of finite, nonempty, pairwise disjoint sets, $E_n$, $n\geq 0$. Moreover, we have maps $r,s:E\rightarrow V$, called {\em range} ($r$) and {\em source} ($s$), such that $r(E_n) = V_n$ and $s(E_n)= V_{n-1}$, $n=1,2,3,\ldots$. The graph is always assumed to have no sources other than $v_0$ and no sinks, that is, $s^{-1}(v)$ and
$r^{-1}(v)$ are non-empty for any $v$ in $V$ (other than $r^{-1}(v_0)$). 
\end{defn}

Let $B = (V, E)$ be a Bratteli diagram, and let $E^*$ be the set of all finite paths in $B$, including the vertices (treated as paths of length zero). That is,

\[
E^* = \{e_1e_2\cdots e_k\mid e_i\in E, r(e_i) = s(e_{i+1})\text{ for }1\leq i\leq k\}\cup V.
\]
The range and source can be extended to $E^*$ by setting
\[
r(e_1e_2\cdots e_k) = r(e_k),\hspace{1cm} s(e_1e_2\cdots e_k) = s(e_1),\hspace{1cm} r(v) = s(v) = v\text{ for all } v\in V.
\]
For $v, w\in V\cup E$, let $vE^*$ denote all the paths starting with $v$, let $E^*w$ be all the paths ending with $w$, and let $vE^*w$ be all the paths starting with $v$ and ending with $w$.

Following \cite[Definition 2.2]{HPS}, for integers $l>k>0$, let 
\[
P_{k,l} = \bigcup_{v\in V_k, w\in V_l}vE^*w.
\]
Then given a strictly increasing sequence $m_n$ of nonnegative integers with $m_0 = 0$, the {\em contraction} of $(V,E)$ to $(m_n)_{n\geq 0}$ is the Bratteli diagram $(V', E')$ where $V'_n = V_{m_n}$, $E_n' = P_{m_{n-1},m_n}$, and $r$ and $s$ are as defined above. 

\begin{defn} \label{equivalentBratteli}(cf \cite[Definition 2.2]{HPS}) Two Bratteli diagrams $(V,E)$ and $(V',E')$ are said to be {\em isomorphic} if there exists a pair of bijections between $V$ and $V'$ and between $E$ and $E'$ preserving the gradings and intertwining the respective range and source maps. We let $\sim$ denote the equivalence relation generated by isomorphism and contraction, and we say that $(V,E)$ and $(V',E')$ are {\em equivalent} if $(V,E)\sim(V',E')$.
\end{defn}

We now recall the notion of an ordered Bratteli diagram.

\begin{defn} An ordered Bratteli diagram $(V, E, \geq)$ is a Bratteli diagram $(V,E)$ together with a partial order $\geq$ on $E$ so that edges $e$ and $e'$ are comparable if and only if $r(e)=r(e')$. 
\end{defn}

If $B= (V, E, \geq)$ is an ordered Bratteli diagram, then for any integers $l>k>0$, $v\in V_l,  w\in V_k$, the set $vEw$ is linearly ordered under the lexicographic ordering. Hence if $(V', E')$ is a contraction of $(V,E)$, it becomes an ordered Bratteli diagram in the natural way, and the notion of equivalence in Definition \ref{equivalentBratteli} extends to ordered Bratteli diagrams.

Since $r^{-1}(v)$ is linearly ordered for each $v\in V\setminus V_0$, each has a maximal element and a minimal element. We let $E_n^{\max}$ (resp. $E^{\min}_n$) denote the set of maximal (resp. minimal) edges at level $n$, and let $E^{\max}$ (resp. $E^{\min}$) denote the set of all maximal (resp. minimal) edges. We can now describe the partial Bratteli-Vershik system associated to an ordered Bratteli diagram, as in \cite{GoncalvesRoyer}.

Let $B = (V,E, \geq)$ be an ordered Bratteli diagram. We define the infinite path space associated to the diagram, denoted by $X_B$, as the following compact subspace (with the product topology) of $\prod\limits_{n=1}^\infty E_n$: 

$$X_B=\left\{\xi\in\prod\limits_{n=1}^\infty E_n\middle| r(\xi_j)=s(\xi_{j+1})\text{ for each  }j\geq 1\right\}.$$

Let $X_1:=X\setminus X_B^{\min}$ and $X_{-1}:=X\setminus X_{\max}$, where

$$X_B^{\max}=\{\xi\in X\mid\xi_j\in E_j^{\max},\ \forall j\geq 1\}$$ 
and 
$$ \ X_B^{\min}=\{\xi\in X\mid\xi_j\in E_j^{\min}, \ \forall j\geq 1\}.$$

To define $\lambda:X_{-1}\rightarrow X_1$ we proceed in the following way: given $\xi\in X_{-1}$, let $n$ be such that $\xi_i \in E_i^{\max}$, for $1\leq i \leq n-1$, and $\xi_n\notin E_n^{\max}$ and define $\lambda(\xi)=g_1...g_{n-1}g_n\xi_{n+1}\xi_{n+2}...\in X_B$, where $g_n$ is the successor of $\xi_n$ and $g_1...g_{n-1}g_n$ is the unique path such that $g_j\in E_j^{\min}$ for each $j\in\{1,...,n-1\}$. We note that $\lambda:X_{-1}\rightarrow X_1$ is a homeomorphism and that this definition does not depend on the number of elements of $X_B^{\max}$ or $X_B^{\min}$. In the sequel, we will call $\lambda$ the {\em Vershik map} associated to $B$.

\subsection{Partial Actions}

\begin{defn} A partial action of a group $G$ on a set $\Omega$ is a pair $\theta=(\{\Delta_{t}\}_{t\in G},$ $\{\theta_{t}\}_{t\in G})$, where for each $t\in G$, $\Delta_{t}$ is a subset of $\Omega$ and $\theta_{t}:\Delta_{t^{-1}} \rightarrow \Delta_{t}$ is a bijection such that:
\begin{enumerate}
	\item $\Delta_{e} = \Omega$ and $\theta_{e}$ is the identity map on $\Omega$;
	\item $\theta_{t}(\Delta_{t^{-1}} \cap \Delta_{s})=\Delta_{t} \cap \Delta_{ts}$ for all $s, t\in G$;
	\item  $\theta_{t}(\theta_{s}(x))=\theta_{ts}(x)$ for all $x \in \Delta_{s^{-1}} \cap \Delta_{s^{-1} t^{-1}}$ and $s, t\in G$.
\end{enumerate}

If $\Omega$ is a topological space, we also require that each $\Delta_{t}$ is an open subset of $\Omega$ and that each $\theta_{t}$ is a homeomorphism of $\Delta_{t^{-1}}$ onto $\Delta_{t}$. 

Analogously, a pair $\theta = (\{ D_{t} \}_{t \in G} , \{ \theta_{t} \}_{t \in G} )$ is a partial action of $G$ on a C*-algebra $A$ if each $D_{t}$ is a closed two sided ideal and each $\theta_{t}$ is a *-isomorphism of $D_{t^{-1}}$ onto $D_{t}$.
\end{defn}

In this work we consider a partial action of $\ZZ$ which is obtained by iteration of a single partially defined homeomorphism as follows. 

\begin{prop}\label{theimportantexample} Let $X$ be a locally compact Hausdorff space, let $U$ and $V$ open subsets of $X$ and let $h$ a homeomorphism from $U$ to $V$. Let $X_{-n}= \text{dom} (h^n)$. Then $\theta=(\{X_{n}\}_{n\in \ZZ},\{h^n\}_{n\in \ZZ})$ is a partial action of $\ZZ$. 
\end{prop}
\begin{proof}
See \cite[Section 3]{Exel}.
\end{proof}
\begin{defn}\label{abbreviation}
Let $X$ be a locally compact Hausdorff space, let $U$ and $V$ open subsets of $X$ and let $h$ a homeomorphism from $U$ to $V$. Then we let $(X, h)$ denote the induced partial action of $\ZZ$ from Proposition \ref{theimportantexample}.
\end{defn}
Suppose that we have $h:U\to V$ as in Proposition \ref{theimportantexample}. Let 
\[
O(x) := \{ h^n(x)\mid n\in \ZZ, h^n(x)\text{ is defined}\};
\]
this is called the {\em orbit} of $x$. We say that $h$ is {\em minimal} if the orbit of $x$ is dense in $X$ for all $x\in X$. We give two examples of minimal partial homeomorphisms.

\begin{ex} {\bf The Odometer}

Let $X=\{0,1\}^{\infty}=\displaystyle \prod_{\NN}\{0,1\}$. Let $\max=1^\infty$ (the sequence of all 1´s), $\min=0^\infty$ (sequence of all 0´s), $X_{-1}=X\setminus\{\max\}$, $X_1=X\setminus\{\min\}$ and $h:X_{-1} \rightarrow X_{1}$ be addition of 1 with carryover to the right. Then $\theta=(\{X_{n}\}_{n\in \ZZ},\{h^n\}_{n\in \ZZ})$, where $X_{-n}= \text{dom}( h^n)$, is a topological partial action. It is straightforward to check that $h$ is minimal. We note that in this case the partial action can be extended to a global action by setting $h(\max) = \min$. 
\end{ex}

\begin{ex}\label{2max1min}
Consider the following ordered Bratteli diagram $B = (V, E, \geq)$:

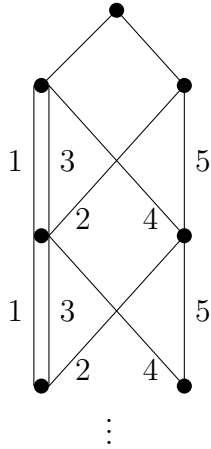
\vspace{1cm}
\begin{center}
\begin{tikzpicture}[scale=2]
\draw (0,0) -- (0,1) node [left,midway]{1};
\draw (0.1, 0) -- (0.1, 1) node [right,midway]{3}; 
\draw (1, 0) -- (1, 1) node [right,midway]{5};
\draw (0.1, 0) -- (1,1)node [below, near start]{2};
\draw (1,0) -- (0.1,1) node [below, near start]{4};

\draw (0,1) -- (0,2 )node [left,midway]{1};
\draw (0.1, 1) -- (0.1,2) node [right,midway]{3}; 
\draw (1, 1) -- (1, 2)node [right,midway]{5};
\draw (0.1, 1) -- (1,2)node [below, near start]{2};
\draw (1,1) -- (0.1,2)node [below, near start]{4};

\draw (0.05, 2) -- (0.55, 2.5);
\draw (1, 2) -- (0.55, 2.5);

\node at (0.05, 1) [shape=circle, fill, inner sep = 0pt, minimum size =2mm]{};
\node at (0.05, 2) [shape=circle, fill, inner sep = 0pt, minimum size =2mm]{};
\node at (0.05, 0) [shape=circle, fill, inner sep = 0pt, minimum size =2mm]{};

\node at (0.55, 2.5) [shape=circle, fill, inner sep = 0pt, minimum size =2mm]{};

\node at (1, 1) [shape=circle, fill, inner sep = 0pt, minimum size =2mm]{};
\node at (1, 2) [shape=circle, fill, inner sep = 0pt, minimum size =2mm]{};
\node at (1, 0) [shape=circle, fill, inner sep = 0pt, minimum size =2mm]{};

\node at (0.5, -0.25) {$\vdots$};
\end{tikzpicture}
\end{center}

The ordering is given by $1\leq 2\leq 3$ and $4\leq 5$, and the diagram and ordering repeat at each level. Here, $X_B^{\max} = \{3^\infty, 5^\infty\}$ and $X_B^{\min} = \{1^\infty\}$. In contrast with the previous example, we cannot extend the Vershik map to all of $X_B$, because $X_B^{\max}$ and $X_B^{\min}$ have different cardinalities. It is also easy to check that $\lambda$ is minimal.
\end{ex}

\begin{rmk} More examples of partial actions can be obtained from full actions. In fact, in \cite{Abadie} it is shown that every partial action on a topological space can be obtained as a restriction of a full action on an envelope space (criteria for when this envelope space is Hausdorff are given in \cite{GiordanoGoncalvesExel}). 
\end{rmk}

\begin{defn}\label{morphism}(cf \cite[Definition 1.1]{Abadie}) Let $\alpha=(\{X_{t}\}_{t\in G},$ $\{\alpha_{t}\}_{t\in G})$ and $\beta=(\{Y_{t}\}_{t\in G},$ $\{\beta_{t}\}_{t\in G})$ be two partial actions of $G$ on the spaces $X$ and $Y$ respectively. Then a {\em morphism} $\phi: \alpha \rightarrow \beta$ is a continuous function $\phi:X \rightarrow Y$ such that for every $t\in G$, $\phi(X_t)\subseteq Y_t$ and the restriction of $\beta_t \circ \phi$ to $X_{t^{-1}}$ equals $\phi\circ \alpha_t$. 
\end{defn}

We now give a criterion for partial actions as in Proposition \ref{theimportantexample} to be isomorphic.

\begin{prop}\label{isopa} Let $X, Y$ be locally compact Hausdorff spaces, $U, V$ be open subsets of $X$, and $W, Z$ be open subsets of $Y$. Let $\alpha$ be a homeomorphism from $U$ to $V$ and $\beta$ be a homeomorphism from $W$ to $Z$. If there exists a homeomorphism $\phi: X\rightarrow Y$ such that $\phi(U) = W$ and $\phi \circ \alpha = \beta \circ \phi$ in $U$ then $(X, \alpha)$ and $(Y, \beta)$ (as in Definition \ref{abbreviation}) are isomorphic. 
\end{prop}
\begin{proof}
In our setting we have that $X_{-n} = \text{dom}(\alpha^n)$, $Y_{-n} = \text{dom}(\beta^n)$, $n\in \ZZ$, $\alpha_n=\alpha^n$ and $\beta_n = \beta^n$. We have to show that, for all $n\in \ZZ$, $\phi(X_n)=Y_n$ and $\phi \circ \alpha^n = \beta^n \circ \phi$ in $X_{-n}$.

We prove first that $\phi \circ \alpha^n = \beta^n \circ \phi$ in $X_{-n}$ for all $n\in \NN$. Notice that if $x\in \text{dom} (\alpha^n)$ then $\alpha^{n-k}(x) \in \text{dom} (\alpha)$ for all $k=1,\ldots, n$. So the usual computation $$\phi \circ \alpha^n(x) = \phi \circ \alpha \circ \alpha^{n-1}(x) = \beta \circ \phi \alpha^{n-1}(x) = \ldots = \beta^n \circ \phi(x)$$ is well defined at every step and we obtain the desired result. Now, since $\alpha^{-1}(V) = U$, $\phi \circ \alpha = \beta \circ \phi$ implies that $\beta^{-1} \circ \phi = \phi\circ \alpha^{-1}$ in $V$ and hence it follows that $\phi \circ \alpha^n = \beta^n \circ \phi$ in $X_{-n}$ for all $n\in \ZZ$.

Next we prove that, for all $n\in \ZZ$, $\phi(X_n)=Y_n$. Notice that $\phi \circ \alpha^n = \beta^n \circ \phi$ implies that $\phi (X_n)\subseteq Y_n$ for all $n\in \ZZ$. For the other inclusion, notice first that $Z=\beta(W)=\beta\circ \phi (U) = \phi \circ \alpha (U) = \phi(V)$. Also, since $U=\phi^{-}(W)$, we have that $\alpha \circ \phi^{-1} = \phi^{-1}\circ \beta$ in $W$ and hence $\alpha^n \circ \phi^{-1} = \phi^{-1}\circ \beta^n$, which implies that $\phi^{-1}(Y_n)\subseteq X_n$ as desired.

\end{proof}

\section{Construction of the Bratteli-Vershik model}

For the rest of this paper, we let $X$ be the Cantor set, and let $U,V\subset X$ be open sets such that 
\[
X_U:= X\setminus U, \hspace{1cm}X_V:=X\setminus V
\]
are nonempty disjoint finite sets. Suppose we have a homeomorphism
\begin{equation}\label{hdef}
h: U\to V
\end{equation}
which is minimal, which we recall means that for all $x\in X$, the orbit $
O(x) := \{ h^n(x)\mid n\in \ZZ, h^n(x)\text{ is defined}\}$
is dense in $X$. We note that if $x\in X$ is such that there exists $k\in \ZZ$ such that $h^k(x)\in X_U$, then $O(x) = \{h^n(x)\mid n\leq k\}$. Likewise, if there instead exists $k\in \ZZ$ such that $h^k(x)\in X_V$, then $O(x) =\{h^n(x)\mid n\geq k\}$.

Our main result, Theorem \ref{maintheo}, is that we can find an ordered Bratteli diagram whose Vershik map is isomorphic (in the sense of Definition \ref{morphism}) to the partial action derived from $h$. This is accomplished with a sequence of lemmas which comprise the rest of the paper.

We begin with a lemma concerning sequences which converge to points in $X_U$ or $X_V$.

\begin{lem}\label{hXU}
\begin{enumerate}
\item Let $x\in X_U$ and suppose that $(x_n)_{n\in\NN}$ is a sequence in $U$ such that $x_n\to x$. Then every accumulation point of $(h(x_n))_{n\in\NN}$ is in $X_V$.
\item Let $x\in X_V$ and suppose that $(x_n)_{n\in\NN}$ is a sequence in $V$ such that $x_n\to x$. Then every accumulation point of $(h^{-1}(x_n))_{n\in\NN}$ is in $X_U$.
\end{enumerate}
\end{lem}
\begin{proof}
We prove the first statement -- the second statement has an analogous proof. Suppose that $x_n\to x\in X_U$, let $y$ be an accumulation point of $(h(x_n))_{n\in\NN}$, and suppose that $y\in V$. Find a subsequence $(h(x_{n_k}))_{k\in\NN}$ which converges to $y$. For all $n\in \NN$, $h(x_n)\in V$, and so we have
\begin{eqnarray*}
h(x_{n_k}) &\to & y\\
h^{-1}(h(x_{n_k})) &\to& h^{-1}(y)\\
x_{n_k}&\to& h^{-1}(y)
\end{eqnarray*}
and so $h^{-1}(y) = x$, which implies that $x\in U$, a contradiction. Hence $y\in X_V$.
\end{proof}
We make the following definitions. For $x\in X$, let 
\[
O^+(x) := \{h^n(x)\mid n\in \NN\},
\] 
\[
O^-(x) := \{h^{-n}(x)\mid n\in \NN\}.
\] 
Also, let
\begin{equation}\label{nU}
n_U(x) = \sup\{n\in \ZZ\mid h^n(x)\in U\} + 1
\end{equation}
%\[
%n_V(x) = \inf\{n\in \ZZ\mid h^n(x)\in V\} + 1
%\]
\begin{lem}\label{densepositiveorbit}
Let $W$ be a clopen set containing $X_V$, and let $x\in X$. Then,
\begin{enumerate}
\item if $O^+(x)$ is infinite, it intersects $W$, and
\item if $x\notin X_V$, then $O^-(x)$ intersects $W$.
\end{enumerate}
\end{lem}
\begin{proof}
We prove the first statement first. Take $x\in X$, suppose that $O^+(x)$ is infinite, and suppose that it does not intersect $W$. Then $(h^n(x))_{n\in\NN}$ is a sequence in the compact set $X\setminus W$, and so has a convergent subsequence $h^{n_k}(x)\to y\in X\setminus W$. Note that $h^n(y)$ cannot be in $X_U$ for any $n\in \NN$, because if we could find such an $n$, by Lemma \ref{hXU} we would be able to find a subsequence of $(h^{n_k+n+1}(x))_{k\in \NN}$ converging to an element of $X_V$, which contradicts our assumption that $O^+(x)$ does not intersect $W$.

Now, find $m\in \ZZ$ such that $h^m(y)\in W$, and find $K$ such that $n_k\geq m$ for all $k\geq K$. Then the sequence $(h^{n_k + m}(x))_{k\geq K}$ is contained in $O^+(x)$ and converges to $h^m(y)\in W$, a contradiction. Hence $O^+(x)$ intersects $W$.

To prove the second statement, suppose $x\notin X_V$. If $O^-(x)$ is finite, say $O^-(x) = \{h^{-n}(x)\mid 1\leq n\leq k\}$, then $h^{-k}(x)\in X_V\subset W$. If $O^-(x)$ is infinite, a similar argument to the above yields that $O^-(x)$ intersects $W$.
\end{proof}

\begin{lem}\label{Kclopen}
Let $W$ be a clopen set containing $X_V$. Then $h^{-1}(W\setminus X_V)\cup X_U$ is clopen. 
\end{lem}
\begin{proof}
Set $K : = h^{-1}(W\setminus X_V)\cup X_U$. To show that this set is closed, we take a sequence $(x_n)_{n\in\NN}$ in $K$ converging to $x\in X$ and prove that $x\in K$. If $x\in X_U$ we are done, so we suppose $x\in U$, and without loss of generality we may assume that no $x_n$ is in $X_U$. Hence $h(x_n)\to h(x)$ and, since $h(x_n)\in W$ for all $n$, we must have that $h(x)\in W$. Since $X_V$ is not in the image of $h$, we must have $h(x) \in W\setminus X_V$ and so $x\in h^{-1}(W\setminus X_V)\subset K$.

We now show that $K$ is open. Since $h^{-1}(W\setminus X_V)$ is open, we need only show that for all $x\in X_U$ we can find an open set around $x$ in $K$. Suppose that this is not possible, that is, every clopen neighborhood of $x$ contains a point outside of $K$. Let $(Z_n)_{n\in\NN}$ be a sequence of clopen sets with intersection $x$, and for each $n$ take $z_n\in Z_n$ with $z_n\in K^c = [h^{-1}(W\setminus X_V)]^c \cap U$. By Lemma \ref{hXU}, there exists a subsequence $(z_{n_k})_{k\in \NN}$ of $(z_n)_{n\in\NN}$ such that $h(z_{n_k})$ converges to a point $y\in X_V$. But because $z_{n_k}\notin h^{-1}(W\setminus K)$ for all $k\in \NN$, and each is in $V$, $h(z_{n_k})$ is not in $W$ for any $k$. This is a contradiction, because this sequence was supposed to converge to $y\in W$ and $W$ is clopen. Hence $K$ is open.
\end{proof}

As stated above, our goal is to construct a sequence of partitions so as to obtain a Bratteli- Vershik model for our partial dynamical system. In previously studied cases of globally defined homeomorphisms, the key ingredient for getting such partitions is a ``first return time'' map which sends each point $x$ in a given clopen set $W$ to the first positive integer such that $h^n(x)$ is back again in $W$. In our setup, there are some points in $X$ whose positive orbit may not intersect a given clopen set at all -- these are the points whose orbit ends in $X_U$. Hence, a first return time may not be defined. We deal with this problem by taking our initial clopen set $W$ to contain $X_V$ and interpreting Lemma \ref{hXU} as saying that ``$h(X_U) = X_V$'', so that if $O^+(x)$ does not intersect $W$ and $h^n(x)\in X_U$, then the return time of $x$ to $W$ ``should be'' $n+1$. 
\begin{prop}\label{returntime}
Let $W\subset X$ be a clopen set containing $X_V$ and disjoint from $X_U$, and define $\tilde r_W: X\to \NN$ by
\[
\tilde r_W(x) =  \begin{cases}\inf \{n\in \NN\mid h^n(x)\in W\} & \text{if }h^n(x)\in W\text{ for some }n\in\NN\\ n_U(x) + 1 & \text{otherwise}\end{cases}
\]
where $n_U$ is as defined in \eqref{nU}. Then $\tilde r_W$ is continuous.
\end{prop}
\begin{proof}
We first note that this map is well-defined by Lemma \ref{densepositiveorbit}.

Let $n\in \NN$ and suppose that $x\in \tilde r_W^{-1}(n)$. We have two cases to consider. The first is to suppose that $h^n(x)\in W$ and $h^k(x)\notin W$ for all $1\leq k < n$. Find a clopen set $W_0$ inside $W\setminus X_V$ containing $h^n(x)$, and let

\begin{eqnarray*}
W_1 &=& h^{-1}(W_0)\cap W^c\\
W_2 &=& h^{-1}(W_1)\cap W^c\\
    &\vdots& \\
W_n &=& h^{-1}(W_{n-1}).
\end{eqnarray*}
Each $W_k$ is clopen and disjoint from $X_U$ and $W$ for $1\leq k < n$. Furthermore, $h(W_k)\subset W_{k-1}$ for all $1\leq k \leq n$, $x\in W_n$ and $h^{n}(W_n)\subset W_0\subset W$. Hence $W_n$ is a clopen set inside $ \tilde r_W^{-1}(n)$ around $x$.

%This is a clopen set which contains $h^{n-1}(x)$, and since $W_1\cap U$ is open, we can find a clopen $Z_1\subset W_1\cap U$ containing $h^{n-1}(x)$. Now let $W_2 = h^{-1}(Z_1)\cap W^c$, and continue in this manner to obtain a clopen set $W_n = h^{-1}(Z_{n-1})\cap W^c$ containing $x$. It is clear that $h(W_k)\subset W_{k-1}$  for all $1\leq k <n$ and that $h(W_1)\subset W$. Since $h^k(W_n)\subset Z_{n-k}\subset W_{n-k}$ and $W_k$ is disjoint from $W$ for all $1\leq k <n$, we must have that $W_n\subset \tilde r_W^{-1}(n)$. 

For the second case, we suppose that $x\in \tilde r_W^{-1}(n)$ and that $n_U(x) +1 = n$, which is to say that $h^{n-1}(x)\in X_U$ and $h^k(x)\notin W$ for any $1\leq k< n-1$. Let $K = h^{-1}(W\setminus X_V)\cup X_U$, and recall from Lemma \ref{Kclopen} that $K$ is clopen. Similarly to the above, let 

\begin{eqnarray*}
W_1 &=& K\cap W^c\\
W_2 &=& h^{-1}(W_1)\cap W^c\\
    &\vdots& \\
W_n &=& h^{-1}(W_{n-1}).
\end{eqnarray*}
Then $h(W_k)\subset W_{k-1}$ for all $2\leq k \leq n$, and each $W_k$ is a nonempty clopen set disjoint from $X_U$ and $W$. For $2\leq k \leq n$, if $y\in W_k$, then $h^{k-1}(y)\in W_1$. If $h^{k-1}(y)\in X_U$, then $h^m(y)$ is not in $W$ for any natural number $m$ and so $\tilde r_W(y) = k$. If on the other hand $h^{k-1}(y)\in h^{-1}(W\setminus V)$, then $h^k(y)\in W$ and $k$ is the first natural number for which this is true, and so again $\tilde r_W(y) = k$. Hence $W_n$ is a clopen set around $x$ contained in $\tilde r_W^{-1}(n)$ and we are done.
\end{proof}
Now that we have a well-defined return time, we can construct the partitions we need to define a Bratteli diagram. 
\begin{prop} \label{KRlemma}Let $h$ be as in \eqref{hdef}, let $W$ be a clopen set containing $X_V$ disjoint from $X_U$, and let $\mathcal{P}$ be a finite partition of $X$ into clopen sets. Then there exist positive integers $K, J(1), J(2),\dots, J(K)$ and a partition $\{E(k, j)\mid 1\leq k\leq K, 1 \leq j \leq J(k)\}$ of $X$ into clopen sets which is finer than $\mathcal{P}$ and which has the following properties: 

\begin{enumerate}
\item $X_U\subset \bigcup_{k = 1}^K E(k, J(k))$, $X_V\subset \bigcup_{k = 1}^K E(k, 1) = W$
\item $h(E(k,j)) = E(k, j+1)$ for $1\leq k\leq K$ and $1\leq j < J(k)$
\item $ \bigcup_{k = 1}^K h(E(k, J(k))\setminus X_U) \cup X_V =   \bigcup_{k = 1}^K E(k, 1)$% for all $1\leq k\leq K$.
\end{enumerate}
\end{prop}
\begin{proof}
Let $W$ be a clopen set containing $X_V$ and disjoint from $X_U$, and let $r_W : = \left.\tilde r_W\right|_W$. Since $r_W$ is continuous and $W$ is compact, its image is a finite set in $\NN$, say $r_W(W) = \{J(1), J(2), \dots, J(K)\}$. For $k = 1, \dots, K$, let $W_k = r_W^{-1}(J(k))$, so that $W = \cup_{k}W_k$ (where this union is a disjoint union). For $k = 1, \dots, K$ and $1\leq j \leq J(K)$, let 

\[E(k, j) = h^{j-1}(W_k).
\] 
This is well-defined because $h^{j-1}(x)$ will not be in $X_U$ for any $x\in W_k$ and $j< J(K)$.

We claim that the family $\{E(k, j)\mid 1\leq k\leq K, 1 \leq j \leq J(k)\}$ satisfies all the conditions given, except possibly that it is finer than $\mathcal{P}$. First we show that the given family of sets is pairwise disjoint, for each member is clearly clopen. Suppose that we have $x\in E(k, j)\cap E(g, l)$, where $1\leq k, g\leq K$, $0\leq j \leq J(K)$, and $1\leq l \leq J(g)$. Without loss of generality, we may assume that $j \leq l$. If $j = l$, then $h^{-j+1}(x)$ is in both $W_k$ and $W_g$, which is only possible if $k = g$. If $j < l$, then $h^{-j+1}(x)\in W_k$, $h^{-l+1}(x)\in W_g$, and $h^{l-j}(h^{-l+1}(x)) = h^{-j}(x)\in W_k\subset W$. Since $l-j < J(g)$, this contradicts the fact that $h^{-l+1}(x)\in W_g$. Hence this family is disjoint.

Before we prove that this family forms a partition of $X$, we prove the numbered statements. For statement 1, if $x\in X_U$, let $k=\min\{j\in \NN: j^{-j}(x)\in W\}$ (notice that k is well defined since $O(x)$ is dense). Then $r_W(h^{-k}(x)) = k+1$, and thus $k+1  = J(l)$ for some $l$ and $h^{-k}(x)\in E(l,1)$. Hence $x = h^k(h^{-k}(x))\in E(l, J(l))$. Hence $X_U\subset \bigcup_{k = 1}^K E(k, J(k))$. The second part of statement 1 is true by the assumption on $W$. Statement 2 is by definition. For statement 3, the containment $\subset$ is straightforward. For the other containment, suppose that $x\in \bigcup_{k = 1}^K E(k, 1)$, and that $x\notin X_V$. Then $h^{-j}(x)\in W$ for some $j\in \NN$ by Lemma \ref{densepositiveorbit}.2 -- we assume that this is the first $j$ for which this is true. Then $j = J(k)$ for some $k$ and $h^{-1}(x) = h^{j-1}(h^{-j}(x)) \in E(k, J(k))\setminus X_U$, proving statement 3.

Finally, we show that this family is a partition of $X$. Let $x\in X$. If $x\in W$ we are done, so suppose that $x\notin W$. %The map $\tilde r_W$ is continuous and hence has bounded range. 
By Lemma \ref{densepositiveorbit}.2, $O^-(x)$ intersects $W$, so let $j$ be the first positive integer such that $h^{-j}(x)\in W$. Then $h^{-j}(x)\in W_k$ for some $k$, and so $x\in h^j(W_k) = E(k, j-1)$. Hence the union of all the $E(k, j)$ is $X$.

%Thus if $h^j(x)\in W$ for only positive powers of $j$, we must have that $O^-(x)$ is finite. Then $h^{-\ell}(x)\in X_V\subset W$ for some $\ell\in \NN$, contradicting the assumption that $h^k(x)\in W$ for only positive powers of $k$. If on the other hand $h^k(x)\in W$ for only negative powers of $k$, then $O^+(x)$ is finite by Lemma \ref{densepositiveorbit} and so $h^\ell(x)\in X_U$ for some $\ell\in \NN$. If $k\in\NN$ is the smallest positive integer such that $h^{-k}(x)\in W$, then $k+\ell+1 \in F$, $h^{-k}(x)\in W_{k+\ell+1}$, and $x\in E(k+\ell+1, k)$.

%So, our last possibility is that there exist $k, \ell \in \NN$ such that $h^{-k}(x), h^\ell(x)\in W$ -- assume that they are the smallest numbers for which this is true. Then $k+\ell\in F$, $h^{-k}(x)\in W_{k+\ell}$, and therefore $x\in E(k+\ell, k)$. No matter which case we are in our point $x$ is in a member of the family $\{E(f, k)\}_{f\in F, 0\leq k< f}$, and so this is a partition of $X$. 

To show that we can achieve the above properties and also be finer than $\mathcal{P}$, one can use a similar argument to that used to prove \cite[Lemma 3.1]{Pu89}.
\end{proof}

\begin{defn}
Let $X$ be the Cantor set and let $h:U\to V$ be as in \eqref{hdef}. We will say that a partition $$\{E(k, j)\mid 1\leq k\leq K, 1 \leq j \leq J(k)\}$$ of $X$ into clopen sets is a {\em Kakutani-Rokhlin partition} for $(X,h)$ if it satisfies conditions 1--3 of Proposition \ref{KRlemma}. For each $k$ with $1\leq k\leq K$, we call the collection $T_k : =\{E(k, j)\mid 1\leq j \leq J(k)\}$ a {\em tower} -- in this case we will say that $E(k,1)$ is the {\em base} of $T_k$ and that $E(k, J(k))$ is the {\em top} of $T_k$. 
\end{defn}

As in \cite[Theorem 4.2]{HPS}, we can obtain a sequence of Kakutani-Rohlin partitions for $(X, h)$ with the aim of producing an ordered Bratteli diagram.

\begin{prop}\label{KRsequence}
Let $X$ be the Cantor set and let $h:U\to V$ be as in \eqref{hdef}. Then for each $n\geq 0 $ there exist $K_n, J(n,1), J(n, 2), \dots J(n, K_n)\in \NN$ and a Kakutani-Rokhlin partition 
\begin{equation}\label{partitionsequence}
\mathcal{P}_n = \{ E(n, k, j)\mid 1\leq k\leq K_n , 1\leq j \leq J(n, k)\}
\end{equation}
of $(X, h)$ such that
\begin{enumerate}
\item $\{\bigcup_{k = 1}^{K_n}E(n, k, 1)\}_{n\geq 0}$ is a decreasing sequence of clopen sets with intersection $X_V$,
\item for all $n\in \NN$, $\mathcal{P}_{n+1}$ is finer than $\mathcal{P}_n$, and
\item $\bigcup_{n\in \NN}\mathcal{P}_n$ generates the topology on $X$.  
\end{enumerate} 
\end{prop}
\begin{proof}
This is proved in the same way as in the analogous result for the minimal homeomorphism case; we apply Proposition \ref{KRlemma} inductively on a decreasing sequence $\{W^{(n)}\}_{n\geq 0}$ of clopen sets disjoint from $X_U$ which converge to $X_V$. For more details, see the proof of \cite[Theorem 4.2]{HPS}.
\end{proof}

From the sequence of Kakutani-Rokhlin partitions \eqref{partitionsequence} we produce an ordered Bratteli diagram, as in \cite{HPS}. To start we let $K_0 = 1$, $J(0,1) = 1$ and $E(0,1,1) = X$. For each $n\geq 0$, we have one vertex in $V_n$ for each tower in $\mathcal{P}_n$, that is
\[
V_n = \{(n,1), (n,2), \dots, (n, K_n)\}.
\]
To define the set of edges $E_n$ from $V_{n-1}$ to $V_n$, we consider how towers in $\mathcal{P}_n$ intersect towers in $\mathcal{P}_{n-1}$. For all $n$ and all $k$ such that $1\leq k \leq K_n$, let $T_{(n,k)}$ be the tower 
\[
T_{(n,k)} = \{E(n, k, j)\mid 1\leq j\leq J(k)\}.
\]
We say that a tower $T_{(n,k)}$ {\em passes through} a tower $T_{(n-1, m)}$ if there exists $j'$ with $1\leq j' \leq J(k)$ such that 
\[
E(n, k, j+j')\subset E(n-1, m, j)\hspace{1cm}\text{for all }1\leq j\leq J(n-1, m)
\]
in $\mathcal{P}_n$. By Proposition \ref{KRsequence}.1 the base of $T_{(n,k)}$ must be contained in the base of some tower in $\mathcal{P}_{n-1}$, and the top of $T_{(n,k)}$ must be contained in the top of some tower in $\mathcal{P}_{n-1}$. Since $\mathcal{P}_{n}$ is a refinement of $\mathcal{P}_{n-1}$, we must have that $T_{(n,k)}$ passes through some number of towers in $\mathcal{P}_{n-1}$. There will be an edge in $E_n$ from the vertex $(n-1,m)$  to the vertex $(n, k)$ for each time $T_{(n,k)}$ passes through $T_{(n-1, m)}$; more formally,
\[
E_n = \{(n, m, k, j')\mid E(n, k, j+j')\subset E(n-1, m, j)\text{ for all }1\leq j\leq J(n-1, m)\}.
\]
From the above, we have that the range and source maps are given by
\[
r(n, m, k, j') = (n,k),\hspace{1cm}s(n,m,k,j') = (n-1, m).
\]
Every edge with range $(n,k)$ has the form $(n,m,k,j')$, and we give all such edges an order based on the fourth coordinate -- this corresponds to ordering the edges based on what order $T_{(n,k)}$ passes through towers in $\mathcal{P}_n$. In symbols, we have
\[
(n,m_1,k,j_1) \geq (n,m_2,k,j_2) \Leftrightarrow j_1 \geq j_2.
\]

From the above, letting $V = \cup V_n$, $E = \cup E_n$, we have that $(V, E, \geq)$ is an ordered Bratteli diagram. This construction mirrors that from \cite[pp. 841--842]{HPS}, where they consider minimal homeomorphisms on the Cantor set which give rise to Bratteli diagrams with one max path and one min path. The following two results mirror \cite[Lemma 4.3]{HPS} and \cite[Theorem 4.4]{HPS}. Their proofs are exactly the same, since the proofs in \cite{HPS} do not rely on uniqueness of max and min path.

\begin{lem}(cf \cite[Lemma 4.3]{HPS})
Let $X$ be the Cantor set, let $h:U\to V$ be as in \eqref{hdef}, let $\{\mathcal{P}_n\}_{n\geq 0}$ be as in \eqref{partitionsequence} satisfying the conditions of Proposition \ref{KRsequence}, let $(V, E, \geq)$ be the ordered Bratteli diagram constructed above, and let $\{m_n\}_{n\geq 0}$ be a strictly increasing sequence of integers with $m_0 = 0$. If $(V', E', \geq')$ is the ordered Bratteli diagram obtained by contracting $(V, E, \geq)$ to the vertices at levels $\{m_n\}_{n\geq 0}$ and $(V'', E'', \geq'')$ is the ordered Bratteli diagram constructed from the sequence $\{\p_{m_n}\}_{n\geq 0}$ of Kakutani-Rokhlin partitions, then $(V', E', \geq')$ and $(V'', E'', \geq'')$ are isomorphic.
\end{lem}

\begin{theo}(cf \cite[Lemma 4.4]{HPS})
Let $X$ be the Cantor set, let $h:U\to V$ be as in \eqref{hdef}, and suppose that $\{\mathcal{P}_n\}_{n\geq 0}$ and $\{\mathcal{Q}_n\}_{n\geq 0}$ are two sequences of Kakutani-Rokhlin partitions satisfying the conditions of Proposition \ref{KRsequence}. Then the ordered Bratteli diagrams constructed from $\{\mathcal{P}_n\}_{n\geq 0}$ and $\{\mathcal{Q}_n\}_{n\geq 0}$ are equivalent.
\end{theo}
Hence, while there is no reason to expect that constructing a Bratteli diagram from $(X,h)$ from two different sequences of clopen sets converging to $X_V$ will result in isomorphic diagrams, they always result in equivalent diagrams. 

We can now prove that for any $B$ constructed from $(X, h)$, the two partial dynamical systems $(X, h)$ and $(X_B,\lambda)$ are isomorphic.

\begin{theo}\label{maintheo}
Let $X$ be the Cantor set, let $h:U\to V$ be as in \eqref{hdef}, and let $B = (V, E, \geq)$ be an ordered Bratteli diagram constructed from $(X, h)$ as above. Let $(X_B, \lambda)$ be the partial dynamical system arising from the Vershik map. Then $(X, h)$ and $(X_B, \lambda)$ are isomorphic.
\end{theo}

\begin{proof}
Let $\{\p_n\}_{n\geq 0}$ be a sequence of Kakutani-Rokhlin partitions as in \eqref{partitionsequence} satisfying the conditions of Proposition \ref{KRsequence}, and let $B = (V, E, \geq)$ be the Bratteli diagram constructed from it. First we describe the infinite path space $X_B$. Let $\{(r, m_r, k_r, j'_r)\}_{r=0}^{n}$ be a finite path in $B$ ending at level $n$. Being a path implies that $k_{n-1} = m_n$ for all $n$. From the definition of the edges, we must have
\[
E(s, k_s, j'_s + j)\subset E(s-1, k_{s-1}, j)\hspace{1cm}\text{for all }j = 1, \dots J(s-1,k_{s-1}),
\]
for $1\leq s\leq n$. In particular, we must have
\[
E(s, k_s, j'_s + j'_{s-1})\subset E(s-1, k_{s-1}, j'_{s-1})
\]
and furthermore
\[
E(s, k_s, j'_s + j'_{s-1}+ j'_{s-2})\subset E(s-1, k_{s-1}, j'_{s-1}+ j'_{s-2}) \subset E(s-2, k_{s-2}, j'_{s-2}).
\]
Inductively, we must have
\begin{equation}\label{path-partition}
E\left(n, k_n, \sum_{s=0}^n j'_s\right)\subset E\left(n-1, k_{n-1}, \sum_{s=0}^{n-1} j'_s\right) \subset\cdots \subset E(0, 1,1) = X. 
\end{equation}
On the other hand, if we start with a partition element as in the left hand side of \eqref{path-partition}, the sequence in \eqref{path-partition} must be unique because $\p_k$ refines $\p_{k-1}$ for all $k$. Hence the map 
\begin{equation}\label{path-partition2}
\{(r, m_r, k_r, j'_r)\}_{r=0}^{n}\mapsto E\left(n, k_n, \sum_{s=0}^n j'_s\right)
\end{equation} 
is a bijection from the set of paths ending at level $n$ in $B$ onto $\p_n$.

Now, let $x\in X$, and consider the set $\xi_x = \{E(n, k, j)\mid x\in E(n, k, j)\}$ which can be identified using \eqref{path-partition2} as a sequence of finite paths starting from the root in $B$. Also, for $n> m$ and $E(n, k_n, j_n), E(m, k_m, j_n)\in \xi_x$, we must have that $E(n, k_n, j_n)\subset E(m, k_m, j_n)$, and so the path corresponding to $E(m, k_m, j_n)$ must a prefix of the path corresponding to $E(n, k_n, j_n)$. Let $\phi(x)$ be the infinite path in $B$ starting at the root determined by the sequence of finite paths $\xi_x$. 

The map $\phi: X \to X_B$ must be injective, because if we have $x,y\in X$ with $x\neq y$, there must be an $n$ such that $x$ and $y$ are in different elements of $\p_n$, since the union of the $\p_n$ generates the topology. Also, given an infinite path $\gamma\in X_B$, the corresponding sequence of finite paths corresponds to a nested sequence of clopen sets in $X$ via \eqref{path-partition2} which must have nonempty intersection. If $x$ is in this intersection, we must have $\phi(x) = \gamma$, so $\phi$ is surjective. Because $\phi$ is built as a correspondence between clopen sets which generate the topology on $X$ and cylinder sets in $X_B$, it is easily seen to be bicontinuous. Hence $\phi$ is a homeomorphism. 

Now, take $x\in U$. If $x$ is in the base of some tower in $\p_n$ for all $n$, we must have that $x\in X_V$. Furthermore, if $x$ is in the top of some tower in $\p_n$ for all $n$, this means that $h(x)$ is in the base of some tower in $\p_n$ for all $n$, hence $h(x)\in X_V$, a contradiction. Hence, $x\in U$ if and only if $x$ is not at the top of a tower at every level. Again, let $\xi_x = \{E(n, k_n, j_n)\mid x\in E(n, k_n, j_n)\}$. Let  $a$ be the first integer for which $x$ is not an element of the top of a tower in $\p_a$, say $x\in E(a, k_a, j_a)$. Since $\xi_x = \{E(n, k_n, j_n)\mid x\in E(n, k_n, j_n)\}$, we must have that 
\[
\phi(x) = (0,1,k_0,j'_0)(1,k_0, k_1, j'_1)\cdots \left(a-1, k_{a-2}, k_{a-1}, j'_{a-1}\right)\left(a, k_{a-1}, k_a, j'_a\right)\cdots
\]
where the equation $\sum_{s=0}^n j'_s = j_n$ holds for all $n \geq 0$. 
From the definition of $a$, the $a$th level is the first level for which $\phi(a)$ does not have a maximal edge. Hence
\[
\lambda(\phi(x)) = (0, 1,m'_0, 0)(1,m'_0,m'_1, 0)\cdots (a-1, m'_{a-2}, m'_{a-1},0)\left(a, m'_{a-1}, k_a, J+j'_a\right)(a+1, k_a, k_{a+1}, j'_{a+1})\cdots
\]
for some integer $J$ and some integers $m'_i$ which are uniquely determined by $\lambda$. The edge $(a, k_{a-1}, k_a, j'_a)$ in $\phi(x)$ indicates that $E(a, k_a, j+j_a)\subset E(a-1, k_{a-1}, j)$ for all $1\leq j \leq J(a-1, k_{a-1})$. Hence, the integer $J$ is the height of $T_{(a-1, k_{a-1})}$, that is, $J = J(a-1, k_{a-1})$. Furthermore, since $E(a-1, k_{a-1}, j_{a-1})$ was assumed to be the top of the tower, we must have 
\begin{equation}\label{jaform}
j_{a-1} = J(a-1, k_{a-1})-1.
\end{equation}

Now we turn to computing $\phi(h(x))$. By definition of $a$ we have $h(x) \in E(a,k_a,j_a+1)$, and moreover
\[
\xi_{h(x)} = \{E(n, k'_n, 1)\}_{n=0}^{a-1} \cup \{E(n, k_n, j_n + 1)\}_{n=a}^{\infty}
\]
where the $k'_n$ above are determined uniquely by $E(a,k_a, j_a +1)$. Hence we must have that 
\[
\phi(h(x)) = (0, 1,k'_0, 0)(1,k'_0,k'_1, 0)\cdots (a-1, k'_{a-1}, k'_{a},0)(a,k'_a, k_a, j_a+1)(a+1,k_{a}, k_{a+1}, j''_{a+1})\cdots
\]
Using \eqref{jaform} we calculate
\[
j_a + 1 = \sum_{s=0}^{a-1} j_s' + j'_a + 1 = j_{a-1} + j'_a  + 1 = j'_a + J(a-1, k_{a-1}).
\]
The fact that $j''_{a+1} + j_a + 1 = j_{a+1} +1$ implies that $j'_{a+1} = j''_{a+1}$, and induction on $i$ shows that $j'_{a+i} = j''_{a+i}$ for all $i\geq 1$. Finally, the initial segments of the paths $\phi(h(x))$ and $\lambda(\phi(x))$ must be the same due to the bijective correspondence given in \eqref{path-partition2}. Hence $\phi(h(x)) = \lambda(\phi(x))$ for all $x\in U$. 

Since $x\in U$ if and only if $x$ is not at the top of a tower at every level, it is easy to see that $\phi(X_U) = X_B^{\max}$, and since $\phi$ is bijective we have $\phi(U) = X_B\setminus X_B^{\max}$. Hence we can apply Proposition \ref{isopa} to conclude that $(X_B, \lambda)$ and $(X, h)$ are isomorphic.
\end{proof}

\begin{rmk} Although the connections of the above model with C*-algebras were not explored in this paper we remark, as previously pointed out in \cite{GoncalvesRoyer}, that the partial crossed product associated to the partial Bratteli-Vershik system, in the case of a simple well ordered diagram, is the "large" AF algebra (tail equivalence) used in the study of Cantor minimal systems and orbit equivalence (see for example \cite{GMPS}). Hence, we expect that our model will be useful in the study of general Bratelli-Vershik systems.
\end{rmk}

\section{Acknowledgements}

The authors would like to thank prof. Ruy Exel for valuable discussions regarding the present paper. In particular we should mention that prof. Ruy Exel has suggested that Theorem~\ref{maintheo} could follow from C*-algebra theory using results of his paper \cite{ExelAF}.

\end{document}